\documentclass[11pt,letterpaper]{article}
\usepackage{amssymb, amsthm, amsmath}
\usepackage{graphicx}
\usepackage{hyperref}

\newtheorem{theorem}{Theorem}
\newtheorem{corollary}[theorem]{Corollary}
\newtheorem{lemma}[theorem]{Lemma}
\newtheorem{proposition}[theorem]{Proposition}

\newtheorem{claim}{Claim}
\newtheorem*{conjecture}{Conjecture}
\theoremstyle{definition}
\newtheorem{case}{Case}

\DeclareMathOperator{\jump}{jump}

\begin{document}
\title{Jumping Sequences}
\author{Steve Butler\thanks{Department of Mathematics, University of California, Los Angeles,
Los Angeles, CA 90095 ({\tt butler@math.ucla.edu}).} \and Ron Graham\thanks{Department of Computer Science and Engineering, University of California, San Diego,
La Jolla, CA 92093 ({\tt graham@ucsd.edu}, {\tt nzang@cs.ucsd.edu}).} \and Nan Zang\footnotemark[2]}
\date{\empty}
\maketitle

\begin{abstract}
An integer sequence $a(n)$ is called a jump sequence if $a(1)=1$ and $1\leq a(n)<n$ for $n\geq 2$.  Such a sequence has the property that  $a^k(n)=a(a(\cdots(a(n))\cdots))$ goes to $1$ in finitely many steps and we call the pattern $(n,a(n),a^2(n),\ldots,a^\ell(n)=1)$ a jumping pattern from $n$ down to $1$.  In this paper we look at jumping sequences which are weight minimizing with respect to various weight functions (where a weight $w(i,j)$ is given to each jump from $j$ down to $i$).  

Our main result is to show that if $w(i,j)=(i+j)/i^2$ then the cost minimizing jump sequence has the property that the number $m$ satisfies $m=a^q(p)$ for arbitrary $q$ and some $p$ (depending on $q$) if and only if $m$ is a Pell number.\end{abstract}

\section{Introduction}\label{sec:intro}
A {\em jump sequence}\/ is an integer sequence $a(n)$ with the following two properties:
\[
\left\{
\begin{array}{l}
a(1)=1\\
1\leq a(n)<n \qquad \mbox{for~} n\geq 2.
\end{array}
\right.
\]
If we use the notation $a^k(n)=\underbrace{a(a(\cdots(a(n))\cdots))}_{k~{\rm terms}}$ with $a^0(n)=n$, then it follows that given a jump sequence and any $n$ that $(a^0(n),a^1(n),a^2(n),a^3(n),\ldots,a^k(n),\ldots)$ forms an initially decreasing sequence which goes from $n$ to $1$ and then stays at $1$ forever after.  We will truncate the tail of $1$s to form a decreasing sequence of integers from $n$ to $1$ and denote this by $\jump(n)$.

While there are many possible jumping sequences, we will be interested in sequences which minimize the cost of going to $1$ with respect to a weighting function.  Given a weight function which assigns for $i<j$ a weight $w(i,j)$, then for any jumping pattern $(b_0{=}n,b_1,b_2,\ldots,b_\ell=1)$ the weight associated with that pattern is the sum
\[
w(b_0,b_1,b_2,\ldots,b_\ell)=\sum_{i=1}^{\ell} w(b_{i-1},b_{i}).
\]

Given a weight function we construct the weight minimizing jumping sequence with respect to the weight by the following procedure: Let $a(1)=1$ and $b(1)=0$.  For $n\geq 2$ let 
\begin{eqnarray*}
b(n)&=&\min_{1\leq k\leq n-1}\big(w(k,n)+b(k)\big);\\
a(n)&=&\mbox{~lowest $k$ which achieves the minimum for $b(n)$.}
\end{eqnarray*}
From the construction it is clear that the resulting sequence $a(n)$ has the property that $w\big(\jump(n)\big)$ gives the minimum possible weight of all jumping patterns between $n$ and $1$.  It is possible that there are several $k$'s which achieve the minimum, by taking the lowest we avoid ambiguity and so more accurately the sequence $a(n)$ finds the lexicographically first jumping sequence which minimizes the weight of all jumping patterns starting at $n$ and decreasing to $1$.  [An example of the ambiguity is for the weight function $w(i,j)=j-i$ where every decreasing sequence between $n$ and $1$ has the same weight.]

Different weight functions lead to different sequences.  As an example let us consider the weight function $w(i,j)=j/i$ (the associated minimizing jump sequence of this function has applications to constructing optimal key trees for networks; see \cite{paper}).
\[
A=(1, 1, 1, 1, 2, 2, 3, 3, 3, 3, 3, 5, 5, 5, 7, 7, 7, 7, 8, 8, 8, 8, 8, 8, 9, 9, 9, 9, 9, 9, 10, 10, 10, 13, 13, \ldots)
\]

Let us remove repetitions and form a new set $\mathcal{A}_1$ which is the set of all numbers which appear in $A$.  Since $A$ is a non-decreasing sequence (as we will show in Section~\ref{sec:basic}) we can easily form this by looking for the next occurrence of a previously unused number.  Doing this we get the following set:
\[
\mathcal{A}_1=\{1, 2, 3, 5, 7, 8, 9, 10, 13, 16, 17, 18, 19, 20, 21, 22, 23, 25, 26, 27, 28, 29, 40, 41, 42,\ldots\}
\]
More generally, we can think of $\mathcal{A}_1$ as all the possible numbers which are $a^1(m)$ for some $m$ (or in other words appear at depth $1$ in some jump pattern).  We can form similar sets $\mathcal{A}_k$ which are composed of all the possible numbers which are $a^k(m)$ for some $m$.  (Again computing $\mathcal{A}_k$ is done by forming the sequence $a^k(n)$ and then looking for the next occurrence of a previously unused term.) If we apply this to our jumping sequence $A$ we get the following first few sets:
\begin{eqnarray*}
\mathcal{A}_0&{=}&\{1,2,3,4,5,6,7,8,9,10,11,12,13,14,15,16,17,18,19,20,21,22,23,24,25,26,\ldots\}\\
\mathcal{A}_1&{=}&
\{1, 2, 3, 5, 7, 8, 9, 10, 13, 16, 17, 18, 19, 20, 21, 22, 23, 25, 26, 27, 28, 29, 40, 41, 42, \ldots\}\\
\mathcal{A}_2&{=}&
\{1, 2, 3, 5, 7, 8, 9, 17, 18, 20, 21, 22, 23, 25, 26, 27, 28, 42, 43, 45, 49, 50, 51, 52, 53,\ldots\}\\
\mathcal{A}_3&{=}&
\{1, 2, 3, 7, 8, 9, 17, 18, 20, 21, 22, 23, 25, 26, 27, 45, 49, 50, 51, 52, 54, 55, 56, 57,59,\ldots\}\\
\mathcal{A}_4&{=}&
\{ 1, 3, 7, 8, 9, 18, 20, 21, 22, 23, 25, 26, 50, 51, 54, 55, 56, 57, 59, 60, 61, 62, 64, 65,\ldots\}\\
\mathcal{A}_5&{=}&
\{ 1, 3, 7, 8, 9, 20, 21, 22, 23, 25, 26, 51, 54, 55, 56, 57, 59, 60, 61, 62, 64, 65, 66, 70,\ldots\}\\
\mathcal{A}_6&{=}&
\{1, 3, 8, 9, 20, 21, 22, 23, 25, 26, 51, 54, 55, 56, 57, 59, 60, 61, 62, 65, 66, 70, 71, 74 ,\ldots\}\\
\mathcal{A}_7&{=}&
\{1, 3, 8, 9, 20, 21, 22, 23, 25, 26, 54, 55, 56, 57, 59, 60, 61, 62, 65, 70, 71, 139, 140 ,\ldots\}\\
\mathcal{A}_8&{=}&
\{ 1, 3, 8, 9, 21, 22, 23, 25, 54, 55, 56, 57, 59, 60, 61, 62, 65, 70, 71, 140, 143, 144,145\ldots\}\\
\mathcal{A}_9&{=}&
\{ 1, 3, 8, 9, 21, 22, 23, 25, 54, 55, 56, 57, 59, 60, 61, 62, 65, 143, 144, 145, 148, 149,\ldots\}\\
\mathcal{A}_{10}&{=}&
\{ 1, 3, 8, 9, 21, 22, 23, 55, 56, 59, 60, 61, 62, 65, 144, 145, 148, 149, 150, 151,\ldots\}
\end{eqnarray*}

There are a few obvious things to note, first is that $\mathcal{A}_k\subseteq\mathcal{A}_{k-1}$, this easily follows from the definition.  The other thing to notice is that the sequence appears to be collapsing to a specific set which we denote
\[
\mathcal{A}_\infty=\bigcap_{k\geq 0}\mathcal{A}_k.
\]
The set $\mathcal{A}_\infty$ is the set of numbers which appear arbitrarily deep in jump patterns.  Pictorially, we can see the sieving process of these sets in Figure~\ref{fig:sieve} where we have marked the numbers which appear in $\mathcal{A}_0,\ldots,\mathcal{A}_{10}$ between $1$ and $100$ and highlighted the members of $\mathcal{A}_\infty$ red.

\begin{figure}[ht]
\centering
\includegraphics[scale=0.43]{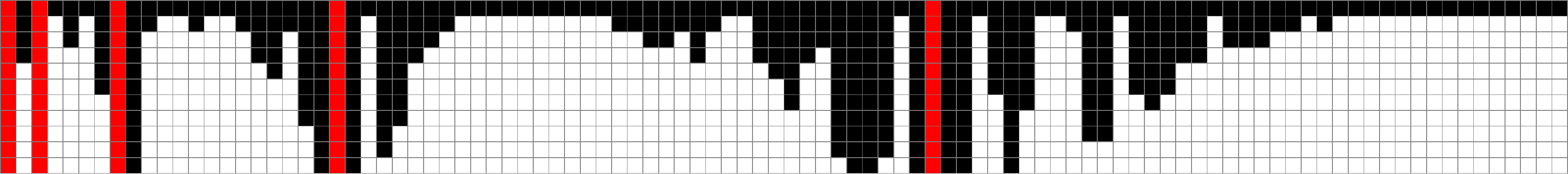}
\caption{The sieving associated with the weight function $w(i,j)=j/i$ and $1\leq n\leq 100$.}
\label{fig:sieve}
\end{figure}

We will proceed as follows.  In Section~\ref{sec:basic} we will establish some basic properties.  Then in Section~\ref{sec:Pell} we will show that for the weight function $w(i,j)=(i+j)/i^2$ the set $\mathcal{A}_\infty$ is the Pell numbers.  In Section~\ref{sec:e} we will look more closely at the weight function $w(i,j)=j/i$.  In Section~\ref{sec:geometric} we will briefly look at finding optimal real jump sequences.  Finally we give some concluding remarks in Section~\ref{sec:conclusion}.

\section{Basic properties of jumping sequences}\label{sec:basic}
To sieve the sequence it is important to know that it is increasing; so that we only need look for the next new term.  We will show for weight functions of a very general nature (including all weight functions that we will consider in this paper) that this is the case.

\begin{proposition}\label{prop:seqincreasing}
Suppose the weight function $w(i,j)$ can be written as
\begin{equation}\label{wform}
w(i,j)=\sum_kf_k(i)g_k(j),
\end{equation}
where each $f_k$ is a {\em non-increasing}\/ function and each $g_k$ is a {\em non-decreasing}\/ function. Then the optimal jump sequence satisfies $a(n-1)\leq a(n)$ for all $n$.  More generally, $a^k(n)$ is non-decreasing for all $k\geq 0$.
\end{proposition}
\begin{proof}
Suppose that $a(n)<a(n-1)$.  Then we must have the following two inequalities:
\begin{eqnarray*}
w\big(a(n-1),n\big)+b\big(a(n-1)\big)&\geq&w\big(a(n),n\big)+b\big(a(n)\big)\\
w\big(a(n),n-1\big)+b\big(a(n)\big)&>&w\big(a(n-1),n-1\big)+b\big(a(n-1)\big)
\end{eqnarray*}
Rearranging the above for $b\big(a(n)\big)-b\big(a(n-1)\big)$ and then combining it follows
\[\begin{array}{c}\vspace{5pt}
w\big(a(n-1),n-1\big)-w\big(a(n),n-1\big)< w\big(a(n-1),n\big)-w\big(a(n),n\big),\qquad\mbox{or}\\
w\big(a(n),n\big)-w\big(a(n-1),n\big)-w\big(a(n),n-1\big)+w\big(a(n-1),n-1\big)<0.
\end{array}
\]
Using \eqref{wform} the left hand side of the inequality can be rewritten as
\[
\sum_k\underbrace{\big(f_k(a(n))-f_k(a(n-1))\big)}_{\geq 0}\underbrace{\big(g_k(n)-g_k(n-1)\big)}_{\geq 0}\geq 0,
\]
giving a contradiction.

The general statement about the sequence $a^k(n)$ then follows by an easy induction.
\end{proof}

The jump sequence $A$ is closely related to the cost sequence $B$ where $b(n)$ is the minimal cost of all possible jump patterns between $n$ and $1$.  Just as the jump function is non-decreasing we also have a similar result for the cost functions (under some very general assumptions about the weight functions which will again include all weight functions that we consider).

\begin{proposition}\label{prop:costincreasing}
Let $w(i,j)>0$ and increasing as a function of $j$.  Then $b(n-1)\leq b(n)$.
\end{proposition}

The proof easily follows by contradiction.  If $b(n)<b(n-1)$ we could construct a sequence jumping from $n-1$ to $1$ with cost less than $b(n)<b(n-1)$; but by definition such a sequence is impossible.

\section{The weight function $w(i,j)=(i+j)/i^2$}\label{sec:Pell}
In this section we will consider using the weight function $w(i,j)=(i+j)/i^2$.  As we did in the introduction we can find the cost minimizing jump sequence
\[
A=(1, 1, 1, 2, 2, 2, 3, 3, 5, 5, 5, 5, 5, 5, 5, 7, 7, 7, 11, 11, 11, 11, 11, 11, 12, 12, 12, 12, 12, 12, \ldots).
\]
If we now sieve this sequence then we see that
\[
\mathcal{A}_7=\{1, 2, 5, 12, 29, 70, 169, 408, 984, 985, 2377, 2378, \ldots\}.
\]
In particular, it looks like the first few elements of $\mathcal{A}_\infty$ will be $1,2,5,12,29,70,169,408$ which correspond to Pell numbers ({\tt A000129} in the {\tt OEIS} \cite{OEIS}).  In this section we will show that this is in fact the case.

\begin{theorem}\label{thm:pell}
Let $w(i,j)=(i+j)/i^2$.  Then $\mathcal{A}_\infty=\{1,2,5,12,29,70,\ldots\}$ is the Pell numbers.
\end{theorem}

The Pell numbers can be defined recursively by $p_1=1$, $p_2=2$ and $p_n=2p_{n-1}+p_{n-2}$ for $n\geq 3$.  It is an easy exercise to show that
\[
p_n={1\over 2\sqrt2}\big(r^n-\bar{r}^n\big),\quad \mbox{where }r=1+\sqrt{2}\mbox{ and }\bar{r}=1-\sqrt{2}.
\]
Alternatively, the Pell numbers can be defined by $p_1=1$ and
$p_n=\big\lfloor rp_{n-1}+0.5\big\rfloor$ for $n\geq 2$.  That is, we can get the next term by multiplying by $r=1+\sqrt{2}$ and then rounding to the nearest integer.  As a consequence of this we have
\begin{equation}\label{eq:pellrounding}
-{1\over2}\leq p_n-rp_{n-1}\leq{1\over2}.
\end{equation}

To establish Theorem~\ref{thm:pell} it will be convenient to think of starting at $1$ and jumping to $n$. The proof will be to show by induction on $k$ that an optimal jump sequence must eventually (i.e., as $n$ gets sufficiently large) start $1=p_1,p_2,\ldots,p_k$, showing that the elements of $\mathcal{A}_\infty$ are the Pell numbers.  

Before starting the proof we need to introduce some constants which will play a role in the analysis.  Let
\[
W=\sum_{k=1}^\infty {p_k+p_{k+1}\over p_k^2}\approx{\tt 5.91570247\ldots}.
\]
By Proposition~\ref{prop:costincreasing} and using the Pell numbers to jump to Pell numbers we have that for any number $n$ that
\[
b(n)\leq b(p_n)\leq \sum_{k=1}^{n-1} {p_k+p_{k+1}\over p_k^2} < W,
\]
so that $W$ is an upper bound for the cost function for all $n$.  (As a result of Theorem~\ref{thm:pell} it will follow that $b(n)\to W$ as $n\to\infty$.)  We will also need to get a bound on the size of the tail of $W$.

\begin{lemma}\label{lem:bound}
$\displaystyle{\sum_{i=2j+1}^\infty {p_i+p_{i+1}\over p_i^2}<{2(2+\sqrt2)\over r^{2j}}+{1\over 8r^{6j}}}.$
\end{lemma}
\begin{proof}
 To estimate the sum on the left we will pair up consecutive terms.  Substituting the definition of the Pell numbers and simplifying we have
\[
{p_{2k+1}+p_{2k+2}\over p_{2k+1}^2}={4(r^{2k+2}+\bar{r}^{2k+2})\over r^{4k+2}+\bar{r}^{4k+2}+2}\mbox{~~~and~~~}
{p_{2k+2}+p_{2k+3}\over p_{2k+2}^2}={4(r^{2k+3}+\bar{r}^{2k+3})\over r^{4k+4}+\bar{r}^{4k+4}-2}.
\]
Combining, we have
\begin{eqnarray}
{p_{2k+1}+p_{2k+2}\over p_{2k+1}^2}+{p_{2k+2}+p_{2k+3}\over p_{2k+2}^2}&=&
{4(r^{2k+2}+\bar{r}^{2k+2})\over r^{4k+2}+\bar{r}^{4k+2}+2}+{4(r^{2k+3}+\bar{r}^{2k+3})\over r^{4k+4}+\bar{r}^{4k+4}-2} \notag \\
&\leq&{4(r^{2k+2}+\bar{r}^{2k+2})\over r^{4k+2}+2}+{4(r^{2k+3}+\bar{r}^{2k+3})\over r^{4k+4}-2} \notag \\
&<&{4(r^{2k+2}+\bar{r}^{2k+2})\over r^{4k+2}}+{4(r^{2k+3}+\bar{r}^{2k+3})\over r^{4k+4}} \label{eq:hardstep} \\
&=&{4\sqrt2r^{2k+4}+4(8-5\sqrt2)\bar r^{2k}\over r^{4k+4}} \notag\\
&<&{4\sqrt2\over r^{2k}}+{1\over 9r^{6k}}. \notag
\end{eqnarray}
The key step in this is \eqref{eq:hardstep} which follows by noting that if $A,B,X, Y-2>0$ then
\[
{A\over X+2}+{B\over Y-2}<{A\over X}+{B\over Y}
\mbox{~~if and only if~~}BX(X+2)<AY(Y-2).
\]
It is simple (though a bit tedious) to verify that this last statement holds in our case.

Using this we now have
\begin{eqnarray*}
\sum_{i=2j+1}^\infty{p_i+p_{i+1}\over p_i^2}&=&
\sum_{k\geq j}\bigg({p_{2k+1}+p_{2k+2}\over p_{2k+1}^2}+{p_{2k+2}+p_{2k+3}\over p_{2k+2}^2}\bigg)\\
&<&\sum_{k\geq j}\bigg({4\sqrt2\over r^{2k}}+{1\over 9r^{6k}}\bigg)\\
&=&{4\sqrt2 r^2\over r^{2j}(r^2-1)}+{r^3\over9r^{6j}(r^3-1)}\\
&<&{2(2+\sqrt2)\over r^{2j}}+{1\over 8r^{6j}}.
\end{eqnarray*}
\end{proof}

The other key ingredient in the proof of Theorem~\ref{thm:pell} is a lower bound on the weight of jump sequences of length $n$ as $n\to\infty$.

\begin{lemma}\label{lemma:pell}
For the weight function $w(i,j)=(i+j)/i^2$, and $a_i$ nondecreasing sequences of real numbers, we have
\[
\lim_{n\to\infty}\bigg( \inf_{1\leq a_1\leq a_2\leq \cdots\leq a_n}w(a_n,\ldots,a_1,1)\bigg)= 3+2\sqrt2.
\]
\end{lemma}
\begin{proof}
Let us denote the infimum of the inside term for a fixed $n$ by $q_n$.  It is easy to see that $q_1=2$ and is achieved by letting $a_1=1$.  For $(a_n,\ldots,a_1,1)$ we have (with some rearranging)
\[
w(a_n,\ldots,a_1,1)=(1+a_1)+{1\over a_1} w(a_n/a_1,\ldots,a_1/a_1)\geq 
(1+a_1)+{1\over a_1}q_{n-1}.
\]
A simple optimization problem then shows that the minimal term on the right is achieved by letting $a_1=\sqrt{q_{n-1}}$ so that
\[
w(a_n,\ldots,a_1,1)\geq 1+2\sqrt{q_{n-1}},
\]
on the other hand it is easy to see that this minimum can be achieved.  So we have the recurrence
\[
q_1=2,\qquad q_n=1+2\sqrt{q_{n-1}}\mbox{~~~for } n\geq 2.
\]
It is an easy exercise to show that $\lim_{n\to\infty}q_n=3+2\sqrt{2}$, establishing the result.
\end{proof}

It should be noted that $\sqrt{3+2\sqrt{2}}=1+\sqrt{2}$ which also plays a role in  the Pell numbers, so that in the limit the jumping sequence with minimal cost is composed of jumps of length $1+\sqrt{2}$;  we will return to this in Section~\ref{sec:geometric}.

\begin{proof}[Proof of Theorem~\ref{thm:pell}]
Clearly the jump sequence must start with $1$, establishing the base case.  Now  assume that eventually all jump sequences must start $1=p_1,p_2,\ldots,p_k$, we need to show that the next term must eventually be $p_{k+1}$.

So suppose that $t$ is a number such that for arbitrary $n$ we have jump sequences of the form $(b_n,\ldots,b_1,t,p_k,\ldots,p_1)$.  Then it must be the case that
\begin{eqnarray*}
W>w\big(\jump(b_n)\big)&=&w(b_n,\ldots,b_1,t,p_k,\ldots,p_1)\\
&=&\sum_{i=1}^{k-1}{p_i+p_{i+1}\over p_i^2}+{p_k+t\over p_k^2}+{1\over t}w(b_n/t,\ldots,b_1/t,1)\\
&\geq&\sum_{i=1}^{k-1}{p_i+p_{i+1}\over p_i^2}+{p_k+t\over p_k^2}+{q_n\over t},
\end{eqnarray*}
where $q_n$ is as defined in Lemma~\ref{lemma:pell}.  Since this holds for all $n$, it also holds for the limit and so $t$ must satisfy
\begin{equation}\notag
W\geq \sum_{i=1}^{k-1}{p_i+p_{i+1}\over p_i^2}+{p_k+t\over p_k^2}+{3+2\sqrt2\over t}.
\end{equation}
Substituting in the definition of $W$ and simplifying this becomes
\begin{equation}\notag
\sum_{i=k+1}^\infty{p_i+p_{i+1}\over p_i^2}\geq{t-p_{k+1}\over p_k^2}+{3+2\sqrt{2}\over t}.
\end{equation}
By the definition of $W$, clearly if we jump along the Pell numbers we have that $t=p_{k+1}$ satisfies this relationship.  The remaining part of the proof is to show that if $t\neq p_{k+1}$ then
\begin{equation}\label{eqn:pellkey}
\sum_{i=k+1}^\infty{p_i+p_{i+1}\over p_i^2}<{t-p_{k+1}\over p_k^2}+{3+2\sqrt{2}\over t}.
\end{equation}

\begin{claim}
It suffices to verify \eqref{eqn:pellkey} does not hold for $t=p_{k+1}\pm1$.
\end{claim}
To see this let
\[
f(t)= {t-p_{k+1}\over p_k^2}+{3+2\sqrt{2}\over t},
\]
then
\[
f'(t)={1\over p_k^2}-{3+2\sqrt2\over t^2}={t^2-(rp_k)^2\over (p_kt)^2}={(t-rp_k)(t+rp_k)\over(p_kt)^2}.
\]
It now follows using \eqref{eq:pellrounding} that for $t\leq p_{k+1}-1$ that $f(t)$ is decreasing and that for $t\geq p_{k+1}+1$ that $f(t)$ is increasing; so if the relationship fails for $t=p_{k+1}\pm1$ then it must fail for every other value of $t$ besides $p_{k+1}$.

\begin{case} $k$ is even and $t=p_{k+1}+1$. \end{case}

By applying Lemma~\ref{lem:bound} to the left hand side of \eqref{eqn:pellkey} it suffices to show that
\[
{2(2+\sqrt2)\over r^{k}}+{1\over 8r^{3k}}<
{1\over p_k^2}+{3+2\sqrt{2}\over p_{k+1}+1}.
\]
Using the definition of Pell numbers along with simplifying this is equivalent to showing
\begin{multline*}
(110\sqrt2+221)r^{6k}+(254\sqrt2+444)r^{5k}+(36\sqrt2+69)r^{4k}\\ 
-(184+124\sqrt2)r^{3k}-(33+18\sqrt2)r^{2k}-(4+2\sqrt2)r^k-1>0.
\end{multline*}
This relationship holds for $k\geq 1$, concluding this case.

\begin{case} $k$ is even and $t=p_{k+1}-1$. \end{case}

By applying Lemma~\ref{lem:bound} to the left hand side of \eqref{eqn:pellkey} it suffices to show that
\[
{2(2+\sqrt2)\over r^{k}}+{1\over 8r^{3k}}<
{-1\over p_k^2}+{3+2\sqrt{2}\over p_{k+1}-1}.
\]
Using the definition of Pell numbers along with simplifying this is equivalent to showing
\begin{multline*}
(110\sqrt2+221)r^{6k}-(254\sqrt2+444)r^{5k}+(36\sqrt2+69)r^{4k}\\+(184+124\sqrt2)r^{3k}-(33+18\sqrt2)r^{2k}+(4+2\sqrt2)r^k-1>0.
\end{multline*}
Again this holds for $k\geq 1$, concluding this case.

\begin{case} $k$ is odd and $t=p_{k+1}+1$. \end{case}

By pulling off the first term and applying Lemma~\ref{lem:bound} to the left hand side of \eqref{eqn:pellkey} it suffices to show that
\[
{p_{k+1}+p_{k+2}\over p_{k+1}^2}+{2\sqrt2\over r^k}+{5\sqrt2-7\over 8r^{3k}}<
{1\over p_k^2}+{3+2\sqrt{2}\over p_{k+1}+1}.
\]
Using the definition of Pell numbers along with simplifying this is equivalent to showing
\begin{multline*}
(5195\sqrt2+7353)r^{10k}-(10566\sqrt2+14920)r^{9k}-(1927\sqrt2+2763)r^{8k}
-(2760\sqrt2+3792)r^{7k}\\
+(66+30\sqrt2)r^{6k}+(832+572\sqrt2)r^{5k}-(62-2\sqrt2)r^{4k}
-(72\sqrt2+48)r^{3k}\\+(23\sqrt2-11)r^{2k}+(8-6\sqrt2)r^k+5\sqrt2-7>0
\end{multline*}
Again this holds for $k\geq 1$, concluding this case.

\begin{case} $k$ is odd and $t=p_{k+1}-1$. \end{case}

By pulling off the first term and applying Lemma~\ref{lem:bound} to the left hand side of \eqref{eqn:pellkey} it suffices to show that
\[
{p_{k+1}+p_{k+2}\over p_{k+1}^2}+{2\sqrt2\over r^k}+{5\sqrt2-7\over 8r^{3k}}<
{1\over p_k^2}+{3+2\sqrt{2}\over p_{k+1}+1}.
\]
Using the definition of Pell numbers along with simplifying this is equivalent to showing
\begin{multline*}
(5195\sqrt2+7353)r^{10k}+(10566\sqrt2+14920)r^{9k}-(1927\sqrt2+2763)r^{8k}+(2760\sqrt2+ 3792)r^{7k}\\
+(66+30\sqrt2)r^{6k}-(832+572\sqrt2)r^{5k}-(62-2\sqrt2)r^{4k}+(72\sqrt2+48)r^{3k}\\
+(23\sqrt2-11)r^{2k}-(8-6\sqrt2)r^k+5\sqrt2-7>0
\end{multline*}
Again this holds for $k\geq 1$, concluding this case and the proof.
\end{proof}

\section{The weight function $w(i,j)=j/i$}\label{sec:e}
One key element in the proof of Theorem~\ref{thm:pell} was that we had a good guess about an optimal way to jump to infinity; and further the total cost of jumping to infinity was finite.  On the other hand for the weight function $w(i,j)=j/i$ we have by the arithmetic-geometric mean inequality that
\[
b(n)=w\big(a^0(n),\ldots,a^k(n)\big)=\sum_{i=1}^{k}{a^{i}(n)\over a^{i+1}(n)}\geq k\sqrt[k]{a^0(n)\over a^k(n)}=kn^{1/k}\geq e\ln n,
\]
so now the weight function is divergent.  On the other hand we can get some control on the weight function, the following result was proved by the authors in \cite{paper}.

\begin{theorem}\label{thm:alpha}
Let $d(n)$ be an integer sequence defined by $d(0)=1$ and $d(n)=\lfloor ed(n) +0.5\rfloor$ for $n\geq 1$ (i.e., multiply by $e$ and then round at each stage); also let
\[
\alpha=\lim_{n\to\infty}\bigg(\sum_{i=1}^n{d(i)\over d(i-1)}-e\ln\big(d(n)\big)\bigg)\approx {\tt 0.01435753\ldots}.
\]
Then $b(n)\leq e\ln n +\alpha + o(1)$.
\end{theorem}

The sequence $d(n)=(1,3,8,22,60,163,443\ldots)$ is {\tt A024581} in the {\tt OEIS} \cite{OEIS}. The definition of this sequence is analogous to the definition of the Pell numbers given in \eqref{eq:pellrounding}.

\begin{corollary}
If $b(n)>e\ln n+\alpha$ then $n\notin \mathcal{A}_\infty$.
\end{corollary}
\begin{proof}
Let $b(n)=e\ln n +\alpha+\varepsilon$ where $\varepsilon>0$.  Then if $n\in \mathcal{A}_\infty$ there would be arbitrarily large $N$ so that $n$ is an element of the jumping pattern of $N$.  Then again using the arithmetic geometric mean it follows that for these $N$
\begin{multline*}
w(N)=w(N=b_0,\ldots,b_k,n,c_0,\ldots,1=c_\ell)=\sum_{i=1}^k{b_{i-1}\over b_{i}}+w(n)\\
\geq e\ln\big({N\over n}\big)+e\ln n+\alpha+\varepsilon=e\ln N +\alpha +\varepsilon.
\end{multline*}
But this contradicts Theorem~\ref{thm:alpha}, so it must be that $n\notin \mathcal{A}_\infty$.
\end{proof}

By definition the elements of the integer sequence introduced in Theorem~\ref{thm:alpha} satisfy $b(d(n))\leq e\ln d(n)+\alpha$, and so can belong to $\mathcal{A}_\infty$.  On the other hand the first integer $n$ that satisfies $b(n)\leq e\ln n+\alpha$ and is {\em not}\/ in the integer sequence is $d(212)+1\approx 1.29131838{\times}10^{92}$.  In particular, we have that the first $212$ elements of $\mathcal{A}_\infty$ come from the integer sequence introduced in Theorem~\ref{thm:alpha}.  This technique though is not sensitive enough to show that $\mathcal{A}_\infty$ is this integer sequence.  The problem seems to be that $ed(n)$ can have fractional part very close to $0.5$ so that we can lose a lot in rounding.  On the other hand for the Pell numbers when we successively multiplied by $1+\sqrt2$ the fractional part went to $0$ and so the ``error'' in rounding rapidly decreased. 

\begin{conjecture}
For the weight function $w(i,j)=j/i$ the set $\mathcal{A}_\infty=\{1,3,8,22,60,\ldots\}$, is composed of the elements in the integer sequence introduced in Theorem~\ref{thm:alpha}.
\end{conjecture}

\section{Jumping to infinity}\label{sec:geometric}
In some cases the set $\mathcal{A}_\infty$ is also the unique increasing sequence which has minimal weight.  In terms of jumping patterns it can be thought of as the jumping pattern to infinity with minimal weight.  An example of this is the Pell numbers for the weight function $w(i,j)=(i+j)/i^2$.  On the other hand, for the function $w(i,j)=j/i$ all such sequences have infinite weight and so there is no ``minimal'' weight sequence.

In this section we will consider constructing optimal jumping patterns to infinity along the real numbers for certain weight functions.  Out main result will be to show that these patterns are geometric.  For a sequence of increasing real numbers $P=(1=a_0,a_1,a_2,a_3,\ldots)$ we let the total weight of $P$ be
\[
w(P)=\sum_{k=0}^\infty w(a_k,a_{k+1}).
\]

\begin{theorem}\label{thm:geometric}
Let $w(i,j)=i^{-\alpha}f(j/i)$ where $\alpha>0$ and with $f(1)>0$, and $f(t)$ non-decreasing, unbounded and continuous for $t\geq1$.  Then there is a $q$ so that for all increasing real sequences $P$,
\[
w(P)\geq {q^\alpha f(q)\over q^\alpha-1}.
\]
Further, equality holds for $P=(1,q,q^2,q^3,\ldots)$.
The value $q$ is one which minimizes $t^\alpha f(t)/(t^\alpha -1)$ for $t\geq 1$.
\end{theorem}

\begin{proof}
Let us first consider a geometric sequence $P=(1,t,t^2,t^3,\ldots)$ for $t>1$.  Then 
\[
w(P)=\sum_{k=0}^\infty{1\over t^{\alpha k}}f(t)=f(t){1\over 1-(1/t^\alpha)}={t^\alpha f(t)\over t^\alpha-1}=g(t).
\]
The function $g(t)$ is positive and continuous for $t>1$; further, $\lim_{t\to1}g(t)=\lim_{t\to\infty}g(t)=\infty$.  In particular, there is some $t$ which minimizes this function which we denote by $q$.  So that for all $t>1$,
\[
{t^\alpha f(t)\over t^\alpha-1}\geq {q^\alpha f(q)\over q^\alpha-1}=Q.
\]
We note that $P=(1,q,q^2,q^3,\ldots)$ has $w(P)=Q$.

The next thing is to show that no sequence has $w(P)<Q$.  To do this we first note that for $P=(1=a_0,a_1,a_2,a_3,\ldots)$ we can rewrite  $w(P)$ as
\begin{eqnarray*}
w(P)=\sum_{k=0}^\infty w(a_k,a_{k+1})&=&\sum_{k=0}^\infty {1\over a_k^\alpha}f\big({a_{k+1}\over a_k}\big)\\
&=&f(a_1)+{1\over a_1^\alpha}\sum_{k=1}^\infty{1\over (a_k/a_1)^\alpha}f\big({a_{k+1}/a_1\over a_k/a_1}\big)=f(a_1)+{1\over a_1^\alpha}w(P_1),
\end{eqnarray*}
where $P_1=(a_1/a_1,a_2/a_1,a_3/a_1,a_4/a_1,\ldots)$.

Now suppose that $w(P_1)\geq w(P)$.  Then the above becomes
\[
w(P)\geq f(a_1)+{1\over a_1^\alpha}w(P)\qquad\mbox{or}\qquad
w(P)\geq {a_1^\alpha f(a_1)\over a_1^\alpha-1}
\geq Q.
\]
It follows for any $P$ with $w(P)< Q$ that $w(P_1)<w(P)$.  Repeating this argument by induction we have for $P_k=(a_k/a_k,a_{k+1}/a_k,a_{k+2}/a_k,a_{k+3}/a_k,\ldots)$ that $w(P_k)<w(P_{k-1})<\cdots<w(P_1)<w(P)$.

Now suppose that $P=(a_0,a_1,a_2,a_3,\ldots)$ satisfies $w(P)<Q$.  We must have that $a_k\to\infty$, otherwise we have $a_k\leq m$ for all $k$ and some $m<\infty$ so that
\[
\sum_{k=0}^\infty {1\over a_k^\alpha}f\big({a_{k+1}\over a_k}\big)
\geq \sum_{k=0}^\infty {1\over m^\alpha}f(1)=\infty.
\]
In particular, there is some $\ell$ so that
\[
{1\over a_\ell^\alpha}Q<Q-w(P).
\]
Now consider the sequence $P'=(a_0,a_1,\ldots,a_\ell,a_\ell q,a_\ell q^2,a_\ell q^3,\ldots)$.  We have that 
\[
w(P')=w(P)-{1\over a_\ell^\alpha}w(P_k) +{1\over a_\ell^\alpha}Q<w(P)+{1\over a_\ell^\alpha}Q<w(P)+Q-w(P)=Q.
\]

Since $w(P')<Q$ we have with the above comments that
\[
Q=w(P_k')<w(P')<Q,
\]
which is impossible.  So our assumption that there is some $P$ with $w(P)<Q$ must be false.  In particular, we have that $w(P)\geq Q$ for all real increasing sequences $P$.
\end{proof}

The weight function $w(i,j)=(i+j)/i^2$ can be rewritten as $w(i,j)=i^{-1}\big(1+(j/i)\big)$.  So applying Theorem~\ref{thm:geometric} with $\alpha =-1$ and $f(t)=1+t$ we have that the optimal jump pattern is $P=(1,r,r^2,r3,\ldots)$ where $r=1+\sqrt{2}$ (which showed up in the Pell numbers); and further the minimal weight is $w(P)=3+2\sqrt2$ which agrees with Lemma~\ref{lemma:pell}.

We cannot apply the results of Theorem~\ref{thm:geometric} directly to $w(i,j)=j/i$.  However, we can consider what happens with the weight function $w(i,j)=j/i^\alpha$ for $\alpha>1$; in this case we have $q=\alpha^{1/(\alpha-1)}$.  If we now let $\alpha\to 1$ then $q\to e$, which we might expect given the construction of the sequence in Theorem~\ref{thm:alpha}.

\section{Conclusion}\label{sec:conclusion}
In this paper we have looked at the construction of cost minimizing jump sequences.  We have seen that when $w(i,j)=(i+j)/i^2$ that the corresponding jump sequence sieves to the Pell numbers.  By choosing different weight functions we can generate different sequences; of course what is interesting are simple weight functions which can generate well known sequences.  Perhaps the most well known sequence is the Fibonacci numbers and these can be generated by a relatively simple weight function.

\begin{theorem}\label{thm:fib}
Let $w(i,j)=(ij+j^2)/i^3$.  Then $\mathcal{A}_\infty=\{1,2,3,5,8,13,\ldots\}$ is the Fibonacci numbers.
\end{theorem}

The proof of Theorem~\ref{thm:fib} is done in the exact same manner as the proof of Theorem~\ref{thm:pell} and we will omit the proof.  Applying the function $w(i,j)=(ij+j^2)/i^3$ to Theorem~\ref{thm:geometric} we get $q=\phi=(1+\sqrt5)/2$; while the Fibonacci numbers can be defined by letting $f_1=1$ and $f_n=\lfloor \phi f_{n-1}+0.5\rfloor$ for $n\geq 2$.

It might appear that finding $\mathcal{A}_\infty$ reduces to finding the optimal real ratio and then ``round; multiply; repeat''.  However this will not always work.  For instance if we let $w(i,j)=(i^2+ij+j^2)/i^3$ then the optimal real ratio is the real root of $2q^3-2q^2-2q-1=0$ which has $q\approx {\tt 1.73990787\ldots}$.  The ``round; multiply; repeat'' would predict the set $\{1,2,3,5,9,16,28,49,\ldots\}$ but instead we have $\mathcal{A}_\infty=\{1, 2, 4, 7, 12, 21, 37, 64,\ldots \}$.  The problem is that instead of $3$ (as we would predict by rounding $3.47981574\ldots$) we got $4$ which then throws off the rest of the sequence.  The problem of finding the set $\mathcal{A}_\infty$ in general remains elusive.  However, it seems reasonable that the technique we have given here could work for  cases when the ratio $q$ is a Pisot-Vijayaraghavan number,  which has the property that $q^n$ is an almost integer while the power of any other root goes to $0$.

We have limited ourselves to homogeneous functions (i.e., functions which can be expressed in the form $w(i,j)=i^{-\alpha}f(j/i)$).  These functions have the benefit that when dealing with real sequences that the tails of the sequences also form cost minimizing sequences, which was used in the proof of Theorem~\ref{thm:pell}.  What can be said about functions which are not homogeneous?  (In \cite{paper} the authors considered the function $w(i,j)=(1-q^j)/i$ where $0<q<1$ is fixed and jump sequences were over real numbers; in that case they showed that the number of jumps in every jump pattern is bounded by a function of $q$.) 

There are also higher dimensional versions of this problem, i.e., jumping in $\mathbb{N}\times\mathbb{N}$ with a weight function between points $w\big((i_1,j_1),(i_2,j_2)\big)$, for various weight functions what can be said about the patterns of the jumps?  

There remain many interesting problems regarding these jumping sequences.

\end{document}